\documentclass[12pt]{article}
\usepackage{amsfonts,amsmath,amssymb,amsthm}
\usepackage{breakcites}
\usepackage{color}
\usepackage{etoolbox}
\usepackage[margin=1in]{geometry}
\usepackage{graphicx}
\usepackage{float}
\usepackage{hyperref}

\graphicspath{{figures/}}

\theoremstyle{plain}
\newtheorem{theorem}{Theorem}[section]
\newtheorem{corollary}[theorem]{Corollary}
\newtheorem{lemma}[theorem]{Lemma}
\newtheorem{proposition}[theorem]{Proposition}

\theoremstyle{definition}

\newtheorem{conjecture}[theorem]{Conjecture}
\newtheorem{question}[theorem]{Question}
\theoremstyle{remark}

\newcommand{\fS}{\mathfrak S}
\newcommand{\cA}{\mathcal A}
\newcommand{\cP}{\mathcal P}
\newcommand{\cQ}{\mathcal Q}

\newcommand\numberthis{\addtocounter{equation}{1}\tag{\theequation}}
\makeatletter
\patchcmd{\@citex}{,}{;}{}{}
\makeatother

\title{Coefficients and roots of peak polynomials
  \footnotetext{Support for this work was
  provided by the National Science Foundation under grants DMS-1062253,
  DMS-1101017, and the University of Washington Mathematics REU 2013 and
  2014.}
}
\author{
  Sara Billey\\
  Department of Mathematics, University of Washington,\\
	Seattle, WA 98195, USA, \texttt{billey@math.washington.edu}
	\and
	Matthew Fahrbach\\
	School of Computer Science, Georgia Institute of Technology,\\
	Atlanta, GA 30332, USA, \texttt{matthew.fahrbach@gatech.edu}
	\and
	Alan Talmage\\
	Department of Mathematics, Penn State University,\\
	University Park, PA 16802, USA, \texttt{abt5217@psu.edu}
}
\date{\today}

\begin{document}
\maketitle

\begin{abstract}
Given a permutation $\pi=\pi_1\pi_2\cdots \pi_n \in \fS_n$, we say an
index $i$ is a peak if $\pi_{i-1} < \pi_i > \pi_{i+1}$.  Let $P(\pi)$
denote the set of peaks of $\pi$.  Given any set $S$ of positive
integers, define ${\cP_S(n)=\{\pi\in\fS_n:P(\pi)=S\}}$.
Billey-Burdzy-Sagan showed that for all fixed subsets of positive
integers $S$ and sufficiently large $n$,
$|\cP_S(n)|=p_S(n)2^{n-|S|-1}$ for some polynomial $p_S(x)$ depending
on $S$.  They conjectured that the coefficients of $p_S(x)$ expanded
in a binomial coefficient basis centered at $\max(S)$ are all
positive.  We show that this is a consequence of a stronger conjecture
that bounds the modulus of the roots of $p_S(x)$.  Furthermore, we
give an efficient explicit formula for peak polynomials in the
binomial basis centered at $0$, which we use to identify many integer
roots of peak polynomials along with certain inequalities and
identities.
\end{abstract}

\section{Introduction} Let $\fS_n$ be the symmetric group of all
permutations $\pi=\pi_1\pi_2\dots\pi_n$ of $[n]:=\{1,2,\dots,n\}$.  An
index $1<i<n$ of $\pi$ is a \textit{peak} if $\pi_{i-1} < \pi_i >
\pi_{i+1}$, and the \textit{peak set} of $\pi$ is defined as
${P(\pi):=\{i : i \text{ is a peak of } \pi\}}$.  We are interested in
counting the permutations in $\fS_n$ with a fixed peak set, so let
$\cP_S(n) := \{\pi\in\fS_n : P(\pi)=S\}$. We say that a set
${S=\{i_1<i_2<\dots<i_s\}}$ is $n$-admissible if $|\cP_S(n)| \neq 0$.
Note that we insist the elements of $S$ be listed in increasing order
and that $S$ is $n$-admissible if and only if $1<i_1$, no two $i_r$
are consecutive integers, and $i_s<n$.  If we make a statement about
an \textit{admissible} set $S$, we mean that $S$ is $n$-admissible for
some $n$, and the statement holds for every $n$ such that $S$ is
$n$-admissible.  Burdzy, Sagan, and the first author recently proved
the following result in \cite{billey}.

\begin{theorem}[{\cite[Theorem 3]{billey}}]\label{MainEnumerationTheorem}
If $S$ is a nonempty admissible set and $m=\max(S)$, then
\[
	|\cP_S(n)|=p_S(n)2^{n-|S|-1}
\]
for $n\geq m$,
where $p_S(x)$ is a polynomial of degree $m-1$ depending on $S$ such
that $p_S(n)$ is an integer for all integral inputs $n$.  If $S=\emptyset$,
then $|\cP_S(n)|=2^{n-1}$ and $p_\emptyset(n)=1$.
\end{theorem}

If $S$ is not admissible, then $|\cP_S(n)|=0$ for all positive
integers $n$, and we define the corresponding polynomial to be
$p_S(x)=0$.  Thus, for all finite sets $S$ of positive integers,
$p_S(x)$ is a well-defined polynomial, which is called
the \textit{peak polynomial} for $S$. 

In this paper, we study properties of peak polynomials such as their
expansions into binomial bases, roots, and related inequalities and
identities.  We also enumerate permutations
with a given peak set using alternating permutations and connect our
results to other recent work about the peak statistic
\cite{billey,signed_peaks,holroyd,kasraoui}.  Our primary motivation
comes from combinatorics, information theory, and probability theory.
Peaks sets have been studied for decades going back to
\cite{kermackmckendrick} and used more recently in a probabilistic
project concerned with mass redistribution \cite{meteors}.  Below are
the principal results of this paper.

\begin{theorem}\label{MasterEnumerationIntro}
Let $S=\{i_1<i_2<\dots<i_s=m\}$ be admissible and nonempty.
For $0 \le j \le m-1$, define the coefficients
\[
	d_{j}^{S} = (-1)^{m-j-1} (-2)^{|S \cap (j, \infty)| - 1} p_{S\cap [j]}(j).
\]
If there exists an index $1 \le r \le s-1$ such that $i_{r+1}-i_r$ is odd,
let $b=i_r$ for the largest such $r$.
Then the peak polynomial $p_S(x)$ expands in the binomial basis centered at $0$ as
\[
	p_S(x) = \sum_{j=b}^{m-1} d_{j}^{S} \binom{x}{j}.
\]
Otherwise, if there are no odd gaps, then
\[
	p_S(x) =  \left(d_{0}^S - (-2)^{|S| - 1}\right) + \sum_{j=1}^{m-1} d_{j}^{S} \binom{x}{j}.
\] 
\end{theorem}

Observe that by Theorem~\ref{MainEnumerationTheorem},  $p_S(m)=0$ using the fact that $\cP_S(m)$ is empty,
but we may have $p_S(\ell)\ne0$ for $\ell<m$ even though
$|\cP_S(\ell)|=0$.  The next two results describe additional roots of $p_{S}(x)$. 

\begin{corollary}\label{OddGapIntro}
If $S = \{i_1 < i_2 < \dots < i_s\}$
and $i_{r+1} - i_r$ is odd for some $1\le r \le s-1$,
then $0, 1, 2, \dots, i_r$ are roots of $p_S(x)$.
\end{corollary}

\begin{theorem}\label{IndexRootIntro}
We have $p_S(i)=0$ for all $i \in S$.
\end{theorem}

Now we discuss two conjectures that inspired this paper.
In the calculus of finite differences,
we define the \textit{forward difference} operator $\Delta$ to be
$(\Delta f)(x) = f(x + 1) - f(x)$.
Higher order differences are given by
${(\Delta^{n}f)(x)=(\Delta^{n-1}f)(x+1)} -{(\Delta^{n-1} f)(x)}$.
We use the definition of the Newton interpolating polynomial to
expand $p_S(x)$ in the binomial basis centered at $k$ as
\[
	p_S(x)=\sum_{j=0}^{m} (\Delta^j p_S)(k) \binom{x-k}{j}.
\]
Notice its similarity to Taylor's theorem.
Below is an example of the forward differences of $p_{\{2,6,10\}}(x)$.
The $k$-th column in the table is the basis vector for the expansion of $p_{\{2,6,10\}}(x)$ in the binomial basis centered at $k$.
In this paper, we consider expansions centered at $0$ and $m$.

\begin{table}[H]\label{difference_table}
  \begin{center}
    \begin{tabular}{|c|ccccccccccc|}
	\hline
$j,k$ & 0 & 1 & 2 & 3 & 4 & 5 & 6 & 7 & 8 & 9 & 10\\
	\hline
0&-8&-4&0&2&4&6&0&-18&-72&-196&0\\
1&4&4&2&2&2&-6&-18&-54&-124&196&3094\\
2&0&-2&0&0&-8&-12&-36&-70&320&2898&12376\\
3&-2&2&0&-8&-4&-24&-34&390&2578&9478&26564\\
4&4&-2&-8&4&-20&-10&424&2188&6900&17086&36376\\
5&-6&-6&12&-24&10&434&1764&4712&10186&19290&33324\\
6&0&18&-36&34&424&1330&2948&5474&9104&14034&20460\\
7&18&-54&70&390&906&1618&2526&3630&4930&6426&8118\\
8&-72&124&320&516&712&908&1104&1300&1496&1692&1888\\
9&196&196&196&196&196&196&196&196&196&196&196\\
10&0&0&0&0&0&0&0&0&0&0&0\\
    \hline
    \end{tabular}
  \end{center}
  \caption{Forward differences of $p_{\{2,6,10\}}(x)$.}
\end{table}

We know from Theorem~\ref{MainEnumerationTheorem} that ${(\Delta^0 p_S)(m)=0}$,
$(\Delta^{m-1} p_S)(k)$ is a positive integer,
and $(\Delta^j p_S)(k)=0$ for all $k \in \mathbb{Z}$ and $j\ge m$.
Burdzy, Sagan, and the first author proposed the following \textit{positivity
conjecture} in \cite{billey}.

\begin{conjecture}[{\cite[Conjecture 14]{billey}}]\label{PositivityConjecture}
Each coefficient $(\Delta^j p_S)(m)$ is a positive integer for ${1\le j\le m-1}$ and all admissible sets $S$.
\end{conjecture}

It follows from Stanley's text \cite[Corollary 1.9.3]{stanley} that $p_S(n)$ is an integer for all integers $n$ if and only if the coefficients in the expansion of $p_S(n)$ in a binomial basis are integers, so we only need to prove that $(\Delta^j p_S)(m)$ is positive for $1 \le j \le m-1$.
In the next section, we show that the positivity conjecture is a consequence of
the following stronger conjecture.

\begin{conjecture}\label{BoundedRootsConjecture}
The complex roots of $p_S(z)$ lie in $\{z\in\mathbb{C} : |z|\le m \text{ and Re}(z)\ge -3 \}$ if $S$ is admissible.
\end{conjecture}

Conjecture~\ref{BoundedRootsConjecture}
is similar in nature to the Riemann Hypothesis.  More
specifically, our work fits into a bigger context of studying roots
for polynomials with integer coefficients in some basis.  For example,
the roots of Ehrhart polynomials
\cite{BDPS,Braun-Develin,Bump-Choi-Kulberg-Vaaler,Pfeifle}, chromatic
polynomials \cite{Brenti-1992, Brenti-Royle-Wagner-1994}, and Hilbert
polynomials \cite{rodriquezvillegas-2002} have all been shown to
respect similar bounds on the complex plane.
Additionally, we are investigating the roots of peak polynomials, because
they may encode properties of their peak set,
similar to how the roots of a chromatic polynomial $P(G,k)$ encode the number of connected components, blocks, and acyclic orientations of $G$.

The paper is organized as follows.
In Section~\ref{PositivityConjectureSection}, we prove
that Conjecture~\ref{BoundedRootsConjecture} implies the positivity
conjecture. Section~\ref{RootsSection} 
proves Theorems~\ref{MasterEnumerationIntro},~\ref{OddGapIntro},
~\ref{IndexRootIntro}, and identifies some special peak polynomials.
Section~\ref{peak_polynomial_at_integers} demonstrates some behaviors
of peak polynomials evaluated at nonnegative integers and patterns in
the table of forward differences of $p_S(x)$.
Section~\ref{ProbabilisticEnumerationSection} develops a new method
for counting the number of permutations with a given peak set using
alternating permutations and the inclusion-exclusion principle.  In
Section~\ref{ConjecturesSection}, we relate our work to other recent
results about permutations with a given peak set.  We conclude with
several conjectures suggested by our investigation.  

\section{An approach to the positivity conjecture}\label{PositivityConjectureSection}

The following lemmas form a chain of arguments that proves that the
positivity conjecture is a consequence of
Conjecture~\ref{BoundedRootsConjecture}.  We write $p(x)$ or $p(z)$ when we are
discussing properties of all polynomials, and we use $p_S(x)$ when we are
discussing peak polynomials in particular.  

\begin{lemma}\label{GaussLucasBound}
If $p(z)$ does not have a complex zero with real part greater than $m$, then
$p'(z),p''(z),\dots,p^{(m-1)}(z)$ do not have a complex zero with real
part greater than $m$, and thus, no real zero greater than $m$.
\end{lemma}

\begin{proof}
We use the Gauss--Lucas theorem, which states that if $p(z)$ is a
(nonconstant) polynomial with complex coefficients, then all the zeros
of $p'(z)$ belong to the convex hull of the set of zeros of $p(z)$.
By assumption all of the roots of $p(z)$ lie in the half-plane
$\{z\in\mathbb{C}:\text{Re}(z) \leq m\}$, so then by the Gauss--Lucas
theorem, all of the roots of $p'(z)$ also lie in this half-plane.
Repeating this argument, we see that $p'(z),p''(z),\dots,p^{(m-1)}(z)$
do not have a complex zero with real part greater than $m$ and thus no
real zero greater than $m$.
\end{proof}

\begin{lemma}\label{NoDerivativeRoots}
If $S$ is
admissible and none of $p_{S}(x),p'_{S}(x),p_{S}''(x),\dots,p_{S}^{(m-1)}(x)$ have a
real zero greater than $m$, then $p_{S}(x),p_{S}'(x),\ldots,p_{S}^{(m-1)}(x)$ are
all positive for $x > m$.
\end{lemma}

\begin{proof}
Since $S$ is admissible, $p_{S}(m+1)$ is a positive integer.  If
$p_{S}(x)$ is nonpositive for some $x_0>m$, then $p_{S}(x)$ has a zero
greater than $m$ by the intermediate value theorem, which contradicts
the assumption.  Therefore $p_S(x)$ is positive for $x>m$, so its
leading coefficient is positive.  It follows that the leading
coefficients of $p_S'(x),p_S''(x),\dots,p_S^{(m-1)}(x)$ are also
positive, so all of the derivatives of $p_{S}(x)$ are eventually
positive.  Again by the intermediate value theorem, the derivatives
$p_S'(x),p_S''(x),\dots,p_S^{(m-1)}(x)$ are all positive for $x>m$.
\end{proof}

We will need the following proposition which we learned from an online
article by Graham Jameson.  Since we don't know of a published version
of this statement, we will include Jameson's proof for the sake of
completeness.

\begin{proposition}[{\cite[Proposition 17]{interpol}}]\label{PositiveForwardDifferences}
For $n \ge 1$,
there exists $\xi \in (x,x+n)$ such that $(\Delta^n p)(x) = p^{(n)}(\xi)$.
\end{proposition}
\begin{proof}
We induct on $n$. When $n=1$, we have the mean value theorem.
Assume the statement is true for a certain $n$.
Then $(\Delta^{n+1}p)(x) = (\Delta^n(\Delta p))(x) = (\Delta^n q)(x)$,
where $q(x) = (\Delta p)(x) = p(x+1)-p(x)$ is a polynomial.
By the induction hypothesis, there exists $\eta \in (x,x+n)$ such that
$(\Delta^n q)(x) = q^{(n)}(\eta) = p^{(n)}(\eta+1) - p^{(n)}(\eta)$.
By the mean value theorem again, this equals $p^{(n+1)}(\xi)$ for some
$\xi \in (\eta, \eta+1)$.
\end{proof}

\begin{lemma}\label{DerivativesToDifferences}
If $p(x)$ is a polynomial of degree $m-1$ and $p'(x),p''(x),\dots,p^{(m-1)}(x)$ are positive for $x>m$,
then all of the forward differences $(\Delta p)(m), (\Delta^2 p)(m), \ldots, (\Delta^{m-1} p)(m)$ are positive.
\end{lemma}

\begin{proof}
There exists $\xi \in (m,m+n)$ such that $(\Delta^n p)(m)=p^{(n)}(\xi)$
using Lemma~\ref{PositiveForwardDifferences}.
By assumption, $p'(x),p''(x),\dots,p^{(m-1)}(x)$ are positive for $x>m$, so
$p'(\xi),p''(\xi),\dots,p^{(m-1)}(\xi)$ are positive since $\xi > m$.
Therefore, $(\Delta p)(m), (\Delta^2 p)(m), \ldots, (\Delta^{m-1} p)(m)$ are positive.
\end{proof}

\begin{theorem}\label{StrongerThanPositivity}
If $S$ is admissible and $p_S(x)$ has no zero whose real part is greater than $m$, then each coefficient $(\Delta^j p_S)(m)$ is positive for $1\le j \le m-1$.
\end{theorem}

\begin{proof}
The proof is a consequence of Lemma~\ref{GaussLucasBound}, Lemma~\ref{NoDerivativeRoots},
and Lemma~\ref{DerivativesToDifferences}.
\end{proof}

It is clear that Conjecture~\ref{BoundedRootsConjecture} satisfies the hypothesis
of Theorem~\ref{StrongerThanPositivity}, so we prove 
Conjecture~\ref{PositivityConjecture}
if we can appropriately bound the roots of $p_S(x)$.

In the supplemental data set \cite{fahrbach}, we used Sage to verify
Conjecture~\ref{PositivityConjecture} and
Conjecture~\ref{BoundedRootsConjecture} for all admissible sets $S$
with $\max(S) \le 15$.
For each row in the table of \cite{fahrbach}, we list a peak set $S$, $p_S(x)$,
the forward differences of $p_S(x)$ centered at $m$, and the complex roots
of $p_S(z)$.
Our Sage code is at the bottom of this document.
We initially computed this data to gain insight about the positivity conjecture,
but after plotting the complex roots of $p_S(z)$,
we conjectured Corollary~\ref{OddGapIntro} and Theorem~\ref{IndexRootIntro}.
We also noticed repeated and predictable structure in the complex roots,
which led to Conjecture~\ref{BoundedRootsConjecture}, Theorem~\ref{Gap3Split}, and Corollary~\ref{Gap3Shift}.

\section{Roots of peak polynomials}\label{RootsSection}

Our main theorems from the introduction are proved here in
Subsection~\ref{sub3-1}.  In particular, we give an explicit formula
for $p_S(x)$ in the binomial basis centered at $0$.  In
Subsection~\ref{sub3-2} we look at peak polynomials with only integral
roots, and the results in Subsection~\ref{sub3-3} show that if $S$ has a gap of
$3$, then $p_S(x)$ is independent of the peaks to the left of this gap
up to a constant.
All of the results in this section assume that $S$ is admissible,
though not explicitly stated in the hypothesis.
Also, note that $m \ne \max(S)$ in most of the recurrences.

\subsection{Main results}\label{sub3-1}
The following recurrence relations are very efficient for computation
and are the foundation of every result in this section.

\begin{corollary}[{\cite[Corollary 4]{billey}}]\label{MainRecursion}
We have
\[
	p_S(x)=p_{S_1}(m-1)\binom{x}{m-1}-2p_{S_1}(x)-p_{S_2}(x),
\]
where $S_1=S \setminus \{m\}$ and $S_2=S_1\cup\{m-1\}$.
\end{corollary}

\begin{lemma}\label{FinalGapRecursion}
If $S=\{i_1<i_2<\dots<i_s=m<m+k\}$ and $k \ge 2$, then
\begin{align*}
	p_S(x) &= -2p_{S_1}(x)\chi(k\text{ even})
						+ \sum_{j=1}^{k-1}(-1)^{k-1-j} p_{S_1}(m+j)\binom{x}{m + j}.
\end{align*}
\end{lemma}

\begin{proof}
We induct on $k$ and use Corollary~\ref{MainRecursion}.
In the base case $k=2$, and
\begin{align*}
	p_S(x) &= -2p_{S_1}(x) + p_{S_1}(m+1)\binom{x}{m+1}.
\end{align*}
By induction, 
\begin{align*}
	p_S(x) &= p_{S_1}(m+k-1)\binom{x}{m+k-1} - 2p_{S_1}(x) - p_{S_2}(x)\\
				&= p_{S_1}(m+k-1)\binom{x}{m+k-1} - 2p_{S_1}(x)\\
				&\phantom{=}\, - \left[ -2p_{S_1}(x)\chi(k-1\text{ even})
						+ \sum_{j=1}^{k-2}(-1)^{k-2-j} p_{S_1}(m+j)\binom{x}{m+j}\right]\\
			  &= -2p_{S_1}(x)\chi(k\text{ even})
			  			+ \sum_{j=1}^{k-1}(-1)^{k-1-j}p_{S_1}(m+j)\binom{x}{m+j}. \qedhere
\end{align*}
\end{proof}

\begin{corollary}\label{CombinatorialFinalGapRecursion}
If $S=\{i_1<i_2<\dots<i_s=m<m+k\}$ and $k\ge2$, then
\[
	|\cP_S(n)| = -\chi(k\text{ even})|\cP_{S_1}(n)| + 
							 \sum_{j=1}^{k-1} (-1)^{k-1-j}
							   \binom{n}{m+j} |\cP_{S_1}(m+j)| \cdot |\cP_\emptyset(n-(m+j))|.
\]
\end{corollary}
\begin{proof}
Apply Theorem~\ref{MainEnumerationTheorem} to
Lemma~\ref{FinalGapRecursion}.
\end{proof}

We can interpret
Corollary~\ref{CombinatorialFinalGapRecursion} combinatorially.
Choose $m+k-1$ of the $n$ elements and arrange them such that their peak set is
$S_1$.
Arrange the remaining $n-(m+k-1)$ elements so that there are no peaks, and
append this sequence to the previous one.
In the combined sequence there is either a peak at $m+k$,$m+k-1$, or no peak after $m$.
Since $m+k\in S$,
\[
	|\cP_S(n)| = \binom{n}{m+k-1}|\cP_{S_1}(m+k-1)| \cdot |\cP_\emptyset(n-(m+k-1))|
						    - |\cP_{S_2}(n)| - |\cP_{S_1}(n)|.
\]
We repeat this procedure for $|\cP_{S_2}(n)|$ to count all the permutations
whose peak set is ${S_1\cup\{m+k-1\}}$, but this also counts permutations
whose peak set is $S_1\cup\{m+k-2\}$ and $S_1$.
We repeat this process until we count permutations whose peak set is
$S_1\cup\{m+1\}$, but this peak set is inadmissible and terminates the procedure.
Notice that $|\cP_{S_1}(n)|$ telescopes because it is included in each iteration
with an alternating sign.

We now present the peak polynomial for a single peak
and the proof of an explicit formula for peak polynomials
with nonempty peak sets in the binomial basis centered at $0$.
The results about roots due to odd gaps and peaks follow.

\begin{theorem}[{\cite[Theorem 6]{billey}}]\label{SinglePeak}
If $S=\{m\}$, then
\[
	p_S(x)=\binom{x-1}{m-1}-1.
\]
\end{theorem}

The following lemma is a special case of the well-known
  Vandermonde identity. The proof is very simple to state in this case,
  so we include it.

\begin{lemma}\label{VandermondeIdentity}
For $m \ge 1$, we have
\[
	\binom{x-1}{m-1} = \sum_{k=0}^{m-1} (-1)^{m-1-k} \binom{x}{k}.
\]
\end{lemma}

\begin{proof}
We induct on $m$. When $m=1$, both terms are 1. Assume the statement is true for
any $m$. Using the induction hypothesis and the standard recurrence,
\[
	\binom{x-1}{(m+1) - 1} = \binom{x}{m} - \binom{x-1}{m-1} =
	\sum_{k=0}^{(m+1)-1} (-1)^{(m+1)-1-k} \binom{x}{k}. \qedhere
\]
\end{proof}

\begin{proof}[Proof of Theorem~\ref{MasterEnumerationIntro}]
The proof follows by iterating Lemma~\ref{FinalGapRecursion}.
In the case that there no odd gaps, we have
\[
	p_S(x) = (-2)^{|S|-1} \left[\binom{x-1}{i_1-1} -1 \right]
						+ \sum_{j=i_1}^{m-1} d_j^S \binom{x}{j},
\]
and then use Lemma~\ref{VandermondeIdentity} to shift the $p_{\{i_1\}}(x)$ term 
to the binomial basis centered at $0$.
\end{proof}

\begin{corollary}\label{OddGapTheorem}
If $S = \{i_1 < i_2 < \dots < i_s\}$
and $i_{r+1} - i_r$ is odd for some $1\le r \le s-1$,
then $0, 1, \dots, i_r$ are roots of $p_S(x)$.
\end{corollary}

\begin{proof}
The proof follows from Theorem~\ref{MasterEnumerationIntro}.
\end{proof}

\begin{corollary}\label{y_intercept}
If $S$ contains an odd peak, then $p_S(0)=0$.
Otherwise, $p_S(0)=(-2)^{|S|}$.
\end{corollary}

\begin{proof}
The proof follows from Theorem~\ref{MasterEnumerationIntro}.
\end{proof}

\begin{theorem}\label{IndexRoot}
We have $p_S(i)=0$ for $i \in S$.
\end{theorem}

\begin{proof}
We induct on $|S|$ for all nonempty admissible sets $S$.
In the base case $|S|=1$, and $p_{\{m\}}(m)=0$ by Theorem~\ref{SinglePeak}.
In the inductive step, let $m=\max(S)$.
If $i\in S_1$, then $p_{S_1}(i)=0$ by the induction hypothesis, so $p_S(i)=0$
by Lemma~\ref{FinalGapRecursion}.
We also know that $p_S(m)=0$ by
Theorem~\ref{MainEnumerationTheorem}, so $p_S(i)=0$ for all $i\in S$.
\end{proof}

\subsection{Peak polynomials with only integral roots}\label{sub3-2}
All of the peak polynomials in this subsection are completely factored and have
all nonnegative integral roots.
As a result, they satisfy Conjecture~\ref{PositivityConjecture} by
Theorem~\ref{StrongerThanPositivity}, because we have bounded the
real part of their roots by $\max(S)$.
In the next two lemmas,
the leading coefficient is all that is recursively defined, and it
depends solely on the structure of ${\{i_1<i_2<\dots<i_s\}}$.
In Conjecture~\ref{all-integral-roots-conjectures-section},
we propose a classification of all peak polynomials with only integral roots.

\begin{lemma}\label{FactoredFinalGap3}
If $S = \{i_1<i_2<\dots<i_s=m<m+3\}$, then
\[
	p_S(x)=\frac{p_{S_1}(m+1)}{2(m+1)!}(x-(m+3))\prod_{j=0}^{m}(x-j).
\]
\end{lemma}

\begin{proof}
Using Lemma~\ref{FinalGapRecursion}, we see that
\begin{align*}
p_S(x) &= 
		\sum_{j=1}^{2}(-1)^{2-j} p_{S_1}(m+j)\binom{x}{m+j}\\
		   &= \frac{\prod_{j=0}^{m}(x-j)}{(m+1)!}\left[\frac{p_{S_1}(m+2)}{m+2}\left(x-\left(m+1+\frac{p_{S_1}(m+1)(m+2)}{p_{S_1}(m+2)}\right)\right)\right],
\end{align*}
but $m+3$ is also a zero of $p_S(x)$ by Theorem~\ref{IndexRoot}.
Equating the two roots, we have
\[
	p_{S_1}(m+2) = \frac{(m+2)p_{S_1}(m+1)}{2},
\]
so then
\begin{align*}
	p_S(x)&=\frac{p_{S_1}(m+1)}{2(m+1)!}(x-(m+3))\prod_{j=0}^{m}(x-j). \qedhere
\end{align*}
\end{proof}

\begin{lemma}\label{FactoredFinalGap3-2}
If $S = \{i_1<i_2<\dots<i_s=m<m+3<m+5\}$, then
\[
	p_S(x)=\frac{p_{S \setminus \{m+3,m+5\}}(m+1)}{12(m+1)!}(x-(m+5))(x-(m+3))(x-(m-2))\prod_{j=0}^{m}(x-j).
\]
\end{lemma}

\begin{proof}
The proof follows from Corollary~\ref{MainRecursion}
and Lemma~\ref{FactoredFinalGap3}.
\end{proof}

The next two corollaries show how $p_S(x)$ grows from $x_0$ to $x_0+1$ for any
$x_0 \in \mathbb{R}$,
and they demonstrate how the roots shift when translating $p_S(x)$ to $p_S(x+1)$.

\begin{corollary}\label{limit1}
If $S = \{i_1<i_2<\dots<i_s=m<m+3\}$, then
\[
	p_S(x+1)= \lim_{t\rightarrow x} \frac{(t+1)(t-(m+2))}{(t-m)(t-(m+3))}p_S(t).
\]
\end{corollary}

\begin{proof}
Write $p_S(x+1)/p_S(x)$ using Lemma~\ref{FactoredFinalGap3} and apply
Theorem~\ref{IndexRoot}.
\end{proof}

\begin{corollary}\label{limit2}
If $S = \{i_1<i_2<\dots<i_s=m<m+3<m+5\}$, then
\[
p_S(x+1)=\lim_{t\rightarrow x}\frac{(t+1)(t-(m-3))(t-(m+2))(t-(m+4))}{(t-(m-2))(t-m)(t-(m+3))(t-(m+5))}p_S(t).
\]
\end{corollary}

\begin{proof}
Write $p_S(x+1)/p_S(x)$ using Lemma~\ref{FactoredFinalGap3-2} and apply
Theorem~\ref{IndexRoot}.
\end{proof}

A limit is needed in Corollary~\ref{limit1} and Corollary~\ref{limit2},
because $p_S(m+1)$ is defined and nonzero by Lemma~\ref{FactoredFinalGap3}
and Lemma~\ref{FactoredFinalGap3-2}, respectively.
We now derive closed-form formulas for $p_S(x)$ when
$S=\{m,m+3,\dots,m+3k\}$ and $S=\{m,m+3,\dots,m+3k,m+3k+2\}$ for $k \geq 1$.
These formulas are direct consequences of Lemma~\ref{FactoredFinalGap3} and
Lemma~\ref{FactoredFinalGap3-2}

\begin{corollary}\label{ConsecutiveGapsOf3Theorem}
If $S=\{m,m+3,\dots,m+3k\}$ for $k \ge 1$, then
\[
	p_S(x) = \frac{(m-1)(x-(m+3k))}{2(m+1)!(12^{k-1})}\prod_{j=0}^{m+3(k-1)}(x-j).
\]
\end{corollary}

\begin{proof}
We induct on $k$.
In the base case, $k=1$ and $S=\{m,m+3\}$.
Using Lemma~\ref{FactoredFinalGap3} and Theorem~\ref{SinglePeak}, we have
\begin{align*}
	p_{\{m,m+3\}}(x) &= \frac{p_{\{m\}}(m+1)}{2(m+1)!}(x-(m+3))\prod_{j=0}^{m}(x-j)\\
	               &= \frac{(m-1)(x-(m+3))}{2(m+1)!}\prod_{j=0}^{m}(x-j).
\end{align*}
In the inductive step, $S=\{m,m+3,\dots,m+3k\}$.
We use Lemma~\ref{FactoredFinalGap3} again, because $p_{S_1}(m+3k-2)$ by the
inductive hypothesis, and it follows that
\begin{align*}
	p_S(x) &= \frac{p_{S_1}(m+3k-2)}{2(m+3k-2)!}(x-(m+3k))\prod_{j=0}^{m+3(k-1)}(x-j)\\
	       &= \frac{(m-1)(m+3k-2)!}{2(m+1)!(12^{k-2})3!}
	          \left[\frac{(x-(m+3k))}{2(m+3k-2)!}\prod_{j=0}^{m+3(k-1)}(x-j)\right]\\
	  &= \frac{(m-1)(x-(m+3k))}{2(m+1)!(12^{k-1})}\prod_{j=0}^{m+3(k-1)}(x-j).\qedhere
\end{align*}
\end{proof}

\begin{corollary}
If $S = \{m, m+3, \dots, m+3k, m+3k+2\}$ for $k\ge1$, then
\[
	p_S(x) = \frac{(m-1)(x-(m+3k+2))(x-(m+3k))(x-(m+3k-5))}
	              {(m+1)!(12^k)}\prod_{j=0}^{m+3(k-1)}(x-j).
\]
\end{corollary}

\begin{proof}
The proof follows from Lemma~\ref{FactoredFinalGap3-2} and
Theorem~\ref{ConsecutiveGapsOf3Theorem}.
\end{proof}

\subsection{Gap of three independence}\label{sub3-3}
The following theorem shows that if $S$ has a gap of three anywhere, then
$p_S(x)$ is independent of the peaks to the left of that gap up to a constant.
Furthermore, the complex roots of $p_S(z)$ depend only on the peaks to the
right of the gap of three and where this gap occurs.
Corollaries of this result follow.

\begin{theorem}\label{Gap3Split}
Let $S_L = \{i_1 < i_2 < \dots < i_\ell=m\}$ and $S_R = \{j_1=2 < j_2 < \dots < j_r\}$.
If $S = \{i_1 < i_2 < \dots < m < m + 3 < (m+1)+j_2 < \dots < (m + 1) + j_r\}$, then
\[
	p_S(x) = \frac{p_{S_L}(m+1)}{2(m+1)!} p_{S_R}(x-(m+1)) \prod_{k=0}^{m}(x-k).
\]
\end{theorem}

\begin{proof}
We first prove the corresponding statement in terms of permutations with
a given peak set.
Fix a positive integer $n > (m+1)+j_r$. 
Choose $m+1$ of the $n$ elements in $[n]$, and arrange them so that their peak set is $S_L$.
Now arrange the remaining $n-(m+1)$ elements so that their peak set is $S_R$.
This construction produces all of the permutations in $\fS_n$ whose peak set is $S$
without repetition, because $m+1$ and $m+2$ cannot be peaks since $m$ and $m+3$ are.
Thus we have
\begin{equation}\label{gap_of_3_via_permutations}
	|\cP_S(n)| = \binom{n}{m+1} |\cP_{S_L}(m+1)| \cdot |\cP_{S_R}(n-(m+1))|.
\end{equation}
Using Theorem~\ref{MainEnumerationTheorem},
\[
	p_S(n)2^{n-|S|-1} = \binom{n}{m+1} p_{S_L}(m+1) 2^{(m+1)-|S_L|-1}
											p_{S_R}(n-(m+1))  2^{(n-(m+1))-|S_R|-1}.
\]
and since $|S| = |S_L| + |S_R|$, we have
\[
	p_S(n) = \frac{p_{S_L}(m+1)}{2(m+1)!} p_{S_R}(n-(m+1)) \prod_{k=0}^{m}(n-k).
\]
This proves the theorem because we have shown that the polynomial on the
right and the left agree on an infinite number of values.
\end{proof}

From the factorization in \eqref{gap_of_3_via_permutations},
we clearly see that $0,1,2,\dots,m$ are zeros of $p_S(z)$,
and the roots of $p_{S_R}(z)$ are roots of $p_S(z)$ when translated to the right
by $m+1$ in the complex plane.
Note that ${\deg(p_S(x)) = m+j_r}$ because $\max (S) = (m+1)+j_r$, but
we also see this by counting the $m+1$ leftmost integer roots
and then the $j_r-1$ roots of $p_{S_R}(x)$.
Theorem~\ref{Gap3Split} also implies Lemma~\ref{FactoredFinalGap3}
when $S_R=\{2\}$ for all $S_L$, because $p_{\{2\}}(x)=x-2$.
The plots and corollaries below demonstrate this independence.

\begin{figure}[H]
\minipage{0.5\textwidth}
	\center
  \includegraphics[width=0.9\linewidth]{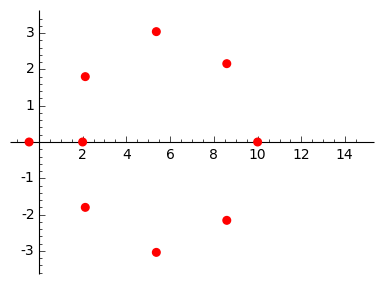}
  \caption{Roots of $p_{\{2,10\}}(z)$.}
\endminipage\hfill
\minipage{0.5\textwidth}
	\center
  \includegraphics[width=0.9\linewidth]{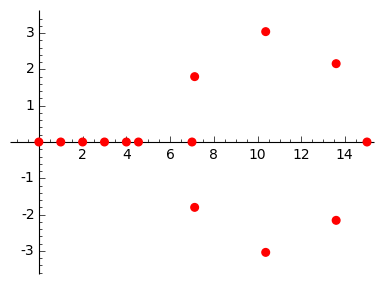}
  \caption{Roots of $p_{\{4,7,15\}}(z)$.}
\endminipage
\end{figure}

\begin{corollary}\label{Gap3RootIndependence}
Let $S_L = \{i_1 < i_2 < \dots < i_\ell=m\}$, $S_R = \{j_1 = 2 < j_2 < \dots < j_r\}$, and
$S = \{i_1 < i_2 < \dots < m < m + 3 < (m+1)+j_2 < \dots < (m + 1) + j_r\}$.
If $S_R$ has no zero with real part greater than $j_r$, then
$p_S(x)$ has no zero with real part greater than $\max (S)$.
\end{corollary}

\begin{proof}
The proof follows from Theorem~\ref{Gap3Split}.
\end{proof}

\begin{corollary}
If $S$ has a gap of three, and $p_{S_R}(x)$ satisfies the positivity conjecture,
then $p_S(x)$ satisfies the positivity conjecture.
\end{corollary}
\begin{proof}
The proof follows from Corollary~\ref{Gap3RootIndependence}
\end{proof}


\begin{corollary}\label{Gap3Shift}
Let $S_L = \{i_1 < i_2 < \dots < i_\ell=m\}$, $S_R = \{j_1 = 2 < j_2 < \dots < j_r\}$, and
$S = \{i_1 < i_2 < \dots < m < m + 3 < (m+1)+j_2 < \dots < (m + 1) + j_r\}$.
If we define $S+1 = \{i+1 : i \in S\}$, then
\[
    p_{S+1}(x) = C(S)p_S(x-1)x,
\]
where
\[
	C(S) = \frac{p_{S_L+1}(m+2)}{(m+2)p_{S_L}(m+1)}
\]
is a constant depending only on $S$.
\end{corollary}

\begin{proof}
Using Theorem~\ref{Gap3Split}, we see that
\[
	p_S(x-1) = \frac{p_{S_L}(m+1)}{2(m+1)!} p_{S_R}(x-(m+2)) \prod_{k=0}^{m}(x-(k+1))
\]
and
\[
	p_{S+1}(x) = \frac{p_{S_L+1}(m+2)}{2(m+2)!} p_{S_R}(x-(m+2)) \prod_{k=0}^{m+1}(x-k).
\]
Solving for $p_{S+1}(x)$, we have
\[
	p_{S+1}(x) = C(S) p_S(x-1) x,
\]
where
\[
	C(S) = \frac{p_{S_L+1}(m+2)}{(m+2)p_{S_L}(m+1)}
\] 
depends only on $S$.
\end{proof}

Observe that Corollary~\ref{Gap3Shift} shifts all of the zeros of $p_S(z)$ 
in the complex plane to the right by one and then picks up a new root at $0$
since $C(S)$ is a constant.
The plots below illustrate this behavior.

\begin{figure}[H]
\minipage{0.5\textwidth}
	\center
  \includegraphics[width=0.9\linewidth]{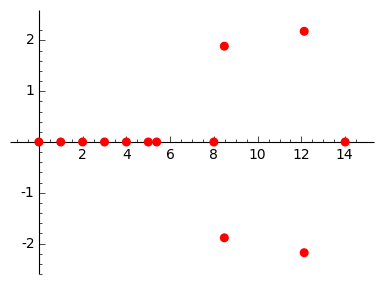}
  \caption{Roots of $p_{\{3,5,8,14\}}(z)$.}
\endminipage\hfill
\minipage{0.5\textwidth}
	\center
  \includegraphics[width=0.9\linewidth]{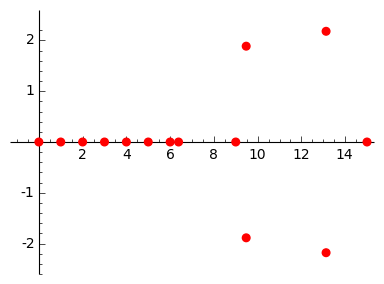}
  \caption{Roots of $p_{\{4,6,9,15\}}(z)$.}
\endminipage
\end{figure}

\section{Evaluating $p_S(x)$ at nonnegative integers}
\label{peak_polynomial_at_integers}

In the previous section, we identified integral roots of $p_S(x)$,
so now we will try to understand the behavior of $p_S(x)$ at
nonnegative integers $j$ when $p_S(j)\ne 0$.
We prove that there is a curious symmetry between column and row $0$ 
in the table of forward differences of
$p_S(x)$ (see Table~\ref{difference_table}),
and that the nonzero values of $|p_S(j)|$ are weakly increasing
for $j \in [\max(S)-1]$ when $\min(S) \ge 4$.
Again, assume that $S$ is a nonempty admissible set in the following hypotheses.

\begin{lemma}\label{alt_binomial_sum}
Let $S\ne\emptyset$ and $m=\max(S)$. For $k\ge0$, we have
\[
	\sum_{j=1}^{k-1}(-1)^{k-1-j} p_S(m + j)\binom{m+k}{m+j} = 
		2p_S(m+k)\chi(k\text{ even}).
\]
\end{lemma}

\begin{proof}
Let $T=S\cup\{m+k\}$.
We know from Theorem~\ref{MainEnumerationTheorem} that
$p_T(m+k)=0$, and then apply Lemma~\ref{FinalGapRecursion}.
\end{proof}

\begin{lemma}\label{p_in_final_gap}
For $S=\{i_1<i_2<\dots<i_s=m<m+k\}$ and $\ell \in [k-1]$,
we have $p_S(m+\ell) = -p_{S_1}(m+\ell)$.
\end{lemma}

\begin{proof}
Using Lemma~\ref{FinalGapRecursion} and Lemma~\ref{alt_binomial_sum}, observe that
\begin{align*}
	p_S(m+\ell) &= -2p_{S_1}(m+\ell)\chi(k \text{ even}) + 
						\sum_{j=1}^{k-1}(-1)^{k-1-j}p_{S_1}(m+j)\binom{m+\ell}{m+j}\\
					&= -2p_{S_1}(m+\ell)\chi(k \text{ even}) + 
						(-1)^{k-\ell}\sum_{j=1}^{\ell-1}(-1)^{\ell-1-j}p_{S_1}(m+j)\binom{m+\ell}{m+j}\\
						&\hspace{5mm} + (-1)^{k-1-\ell}p_{S_1}(m+\ell)\\
					&= -2p_{S_1}(m+\ell)\chi(k \text{ even}) + 
						(-1)^{k-\ell}2p_{S_1}(m+\ell)\chi(\ell\text{ even}) + (-1)^{k-1-\ell}p_{S_1}(m+\ell).
\end{align*}
Considering all possible parities of $k$ and $\ell$,
we see that $p_S(m+\ell)= -p_{S_1}(m+\ell)$.
\end{proof}

\begin{theorem}\label{symmetry}
Let $S\ne\emptyset$ and $m=\max(S)$.
If $j\in \{0,1,\dots,m\}$, then
\[
	(\Delta^j p_S)(0) = (-1)^{m+j} p_S(j).
\]
\end{theorem}

\begin{proof}
We induct on $|S|$.
In the base case $|S|=1$, and we use Theorem~\ref{SinglePeak} and 
Lemma~\ref{VandermondeIdentity}.
\ifx
Vandermonde's identity to observe
\begin{align*}
	p_{\{m\}}(x) &= \left[\sum_{j=0}^{m-1} \binom{-1}{m-1-j}\binom{x}{j} \right] - 1.
\end{align*}
\fi
It follows that
\begin{align*}
	(\Delta^j p_{\{m\}})(0) &=
		\begin{cases}
			(-1)^{m-1} - 1 & \text{if } j = 0,\\
			(-1)^{m-1-j} & \text{if } j \in [m-1],\\
			0 & \text{if } j=m.\\
		\end{cases}
\end{align*}
Similarly, we use Theorem~\ref{SinglePeak} to evaluate
\begin{align*}
	(-1)^{m+j} p_S(j) &= (-1)^{m+j}\left[\binom{j-1}{m-1} -1 \right] \\
		&=
		\begin{cases}
			(-1)^{m+1} - 1 & \text{if } j = 0,\\
			(-1)^{m+j+1} & \text{if } j \in [m-1],\\
			0 & \text{if } j=m,\\
		\end{cases}
\end{align*}
which proves the base case.

In the inductive step $|S| \ge 2$, so let $S=\{i_1<i_2<\dots<i_s=m<m+k\}$ for $k\ge2$.
Using Lemma~\ref{FinalGapRecursion} and expanding $p_{S_1}(x)$ in the
binomial basis centered at $0$,
\begin{align*}\
p_S(x) &= -2p_{S_1}(x)\chi(k\text{ even})
						+ \sum_{j=m+1}^{m+k-1}(-1)^{k-1-(j-m)} p_{S_1}(j)\binom{x}{j}\\
			 &= -2\left[
			 				\sum_{j=0}^m (\Delta^j p_{S_1})(0) \binom{x}{j}
						\right]\chi(k\text{ even})
			 			+ \sum_{j=m+1}^{m+k-1}(-1)^{k-1-(j-m)} p_{S_1}(j)\binom{x}{j}
						\numberthis \label{to_binomial_basis}.
\end{align*}
Assume the case that $j \in \{0,1,\dots,m\}$.
Considering both possible parities of $k$,
we use \eqref{to_binomial_basis} and the induction hypothesis
to see that
\begin{align*}
	(\Delta^j p_S)(0) &= -2(\Delta^j p_{S_1})(0)\chi({k\text{ even}})\\
									&= -2(-1)^{m+j}p_{S_1}(j) \chi({k\text{ even}})\\
									&= (-1)^{(m+k)+j} p_S(j),
\end{align*}
because $p_S(j)=-2p_{S_1}(j)\chi(k\text{ even})$ by Lemma~\ref{FinalGapRecursion}.
Now let $j \in \{m+1,m+2,\dots,m+k-1\}$.
Using Lemma~\ref{p_in_final_gap} and \eqref{to_binomial_basis}, we have
\begin{align*}
	(\Delta^j p_S)(0) &= (-1)^{k-1-(j-m)}p_{S_1}(j)\\
									&= (-1)^{(m+k)+j}p_S(j).
\end{align*}
Lastly, $(\Delta^m p_S)(0) = 0$ because $\deg(p_S(x))=m-1$, which completes the proof.
\end{proof}

For example, if $j>0$ is between the largest odd gap and $m$, then by this
symmetry property and Theorem~\ref{MasterEnumerationIntro} one can observe that
\[
	p_S(j) = (-1)^{m+j}(\Delta^j p_S)(0) = -(-2)^{|S \cap (j,\infty)| - 1}
		p_{S\cap[j]}(j).
\]
If $S$ has no odd gaps, then the equation above holds for all $j\in[m]$.

\begin{lemma}\label{strictly_increasing_after_m}
If $S\ne\emptyset$ and $m=\max(S)$, then $p_S(j) < p_S(j+1)$ for $j \ge m$.
\end{lemma}

\begin{proof}
We prove the result by splitting into two cases.
When $|S|=1$, we have $p_{\{m\}}(x)$, which increases on $(m-1,\infty)$
by Theorem~\ref{SinglePeak} and proves our claim.
In the second case, let $|S|\ge 2$.
We want to show that $p_S(j) < p_S(j+1)$, which is equivalent to showing
${2|\cP_S(j)| < |\cP_S(j+1)|}$,
so we need to construct more than twice as many permutations in $\fS_{j+1}$
with peak set $S$ than there are in $\fS_j$.
Note that $p_S(m)=0$ and $p_S(m+1)>0$, so we need only consider $\fS_{j}$ for
$j \ge m+1$.
First, let $\pi \in \fS_j$ and append $j+1$ to $\pi$.
This gives us $|\cP_S(j)|$ permutations in $\fS_{j+1}$.
Now construct $|\cP_S(j)|$ different permutations by inserting $j+1$
between positions $m-1$ and $m$, so that $j+1$ becomes the final peak.
Lastly, place $j+1$ at the first peak position (reading left to right),
$j$ at the next peak position, etc., and then fill the empty indices from left
to right with $1,2,\dots,j+1-|S|$, respectively.
Each of the $2|\cP_S(n)| + 1$ constructed permutations is distinct and has peak set
$S$, so $p_S(j) < p_S(j+1)$.
\end{proof}

\begin{theorem}\label{increasing_coefficients}
  Let $S=\{i_1<i_2<\dots<i_s=m\}$.  For integers $1 \le j <k$, we
  have $|p_S(j)| \le |p_S(k)|$ provided $p_S(k) \ne 0$,
  except for the case $\{2\} \subsetneq S$ where $p_S(1)=2p_S(3)=-(-2)^{|S|-1}$.  
\end{theorem}

\begin{proof}
  If $|p_S(j)|=0$, then the claim is trivially true, so assume that
  $|p_S(j)| > 0$ which implies $S \cap (j,\infty)$ has no odd gaps.
  If $S=\emptyset$ or not admissible then the statement holds so assume $S\neq
  \emptyset$, admissible,  and $m=\max(S)$.  We first consider the cases where
  $j<k<m$.  We use these assumptions along with
  Theorem~\ref{MasterEnumerationIntro} and Corollary~\ref{symmetry} to
  observe that
\begin{equation}\label{eq:relate}
	|p_S(j)| = 2^{|S\cap(j,\infty)| - 1} |p_{S\cap[j]}(j)|.
\end{equation}

Consider the case $p_S(j+1) \ne 0$.  Then $j+1 \not\in S$ by
Theorem~\ref{IndexRoot}, and
\begin{align*}
	|p_S(j+1)| &= 2^{|S\cap(j+1,\infty)| - 1} |p_{S\cap[j+1]}(j+1)|\\
						 &= 2^{|S\cap(j,\infty)| - 1} |p_{S\cap[j]}(j+1)|.
\end{align*}
To show that $|p_S(j)| \le |p_S(j+1)|$ it suffices to show that
$|p_{S\cap[j]}(j)| \le |p_{S\cap[j]}(j+1)|$.  If $S\cap[j] =
\emptyset$, then we know $p_\emptyset(x)=1$ from
Theorem~\ref{MainEnumerationTheorem}.  Otherwise, we may use
Lemma~\ref{strictly_increasing_after_m} because $S \ne \emptyset$ and
$j \ge \max(S \cap [j])$.  In both cases, $|p_S(j)| \le |p_S(j+1)|$ when
$|p_S(j+1)|>0$.

Now assume that $p_S(j+1)=0$.  Combining
Theorem~\ref{MainEnumerationTheorem}, Corollary~\ref{symmetry}, and the
assumption that $|p_S(j)|>0$, this implies $|p_{S\cap[j+1]}(j+1)|=0$
which in turn implies $j+1\in S$ by
Lemma~\ref{strictly_increasing_after_m}.  Since $S$ is admissible $j+2
\not \in S$ so $p_{S\cap[j+1]}(j+2) = p_{S\cap[j+2]}(j+2) > 0$.  By
\eqref{eq:relate} this implies $|p_S(j+2)| > 0$.  To show that
$|p_S(j)| \le |p_S(j+2)|$, we will show that
\begin{equation}\label{eq:inequality}
	2^{|S\cap(j,\infty)| - 1} |p_{S\cap[j]}(j)| \le
	2^{|S\cap(j+2,\infty)| - 1} |p_{S\cap[j+2]}(j+2)|,
\end{equation}
assuming $j+1 \in S$.  Let $R = S\cap[j+2]$, and $R_1=R\setminus\{j+1\}$.    Using
Theorem~\ref{MainEnumerationTheorem}, \eqref{eq:inequality} is true if
and only if
\begin{equation}\label{eq:inequality.2}
	4 |\cP_{R_1}(j)| \le |\cP_R(j+2)|.
\end{equation}
To prove \eqref{eq:inequality.2}, observe that one can choose any $j$
elements from $[j+1]$, arrange them to have peak set $R_1$ in
$|\cP_{R_1}(j)|$ ways, and then append $j+2$ and the remaining element
to this sequence in decreasing order.  The resulting permutation has
peak set $R$, and doing this in all possible ways yields
$(j+1)|\cP_{R_1}(j)|$ distinct permutations in $\fS_{j+2}$.  If $j+1
\ge 4$, then \eqref{eq:inequality.2} holds so $|p_S(j)| \le |p_S(k)|$
when $|p_S(j+1)|=0$.  Observe that the exact same argument proves the
theorem for the case $m>3$, $j=m-1$, and $k=m+1$.

If $j+1 \in \{2,3\}$, then by \eqref{eq:relate} we
can complete the proof using the fact that $p_\emptyset (x)=1$, and by
computing the values of $p_{\{2\}}(n)$ and $p_{\{3\}}(n)$ for $n=0,1,2,3,4$,
we have
$$S=\{2\} \implies (-2, -1, 0, 1, 2)$$   
and       	
$$S=\{3\} \implies (0, -1, -1, 0, 2).  $$  
In fact, using that data and Theorem~\ref{MasterEnumerationIntro} we
see $p_{S}(1)=-(-2)^{|S|-1}$ for all nonempty admissible sets $S$
with no odd gaps and $0$ otherwise.  Similarly,
$$
p_S(2) = 
\begin{cases}
  0  & \text{if } 2 \in S \text{ or $S$ has an odd gap},\\
  1 & \text{if } S=\emptyset,\\
 -(-2)^{|S|-1} & \text{otherwise,}
\end{cases}
$$
and
$$
p_{S}(3)=\begin{cases}
  0  & \text{if } 3 \in S \text{ or $S$ has an odd gap after } 3,\\
  1 & \text{if } S \subset [2], \\
 -(-2)^{|S|-2} & \text{if } \{2\} \subsetneq S,\\
 -(-2)^{|S|-1} & \text{otherwise,}
\end{cases}
$$
which proves the special case of the theorem where the inequality does
Not hold. For completeness, 
$$
p_{S}(4)=\begin{cases}
  0  & \text{if } 4 \in S \text{ or $S$ has an odd gap after } 4,\\
  1 & \text{if } S =\emptyset ,\\
  2 & \text{if } S=\{2\} \text{ or } S=\{3\},\\
 - (-2)^{|S|-1} & \text{if } \{2,3\} \cap S = \emptyset, |S|>1, \text{ and } S \text{ has no odd gaps},\\
  (-2)^{|S|-1} & \text{otherwise.}\\
\end{cases}
$$
For $n>4$, the values of $|p_S(n)|$ are not typically powers of $2$.  

Finally, the theorem holds for all remaining cases with $m<j<k$ by
Lemma~\ref{strictly_increasing_after_m} and transitivity.   
\end{proof}

The previous proof also implies the following statement.

\begin{corollary}\label{cor:more.zeros}
Let $S$ be a set of positive integers and $j$ be a positive integer such
that $p_{S}(j)\neq 0$.  Let $k\geq j$ integer.  If $p_{S}(k)=0$, then $k
\in S$.
\end{corollary}

\section{Connections to alternating permutations}
\label{ProbabilisticEnumerationSection}
In this section, we enumerate permutations with a given peak set using
alternating permutations and tangent numbers instead of the
recurrence given by Lemma~\ref{FinalGapRecursion}.
Alternating permutations allow us to easily count the number of permutations
whose peak set is a superset of $S$, so we combine this idea 
with the inclusion-exclusion principle to evaluate $|\cP_S(n)|$.

Assume that $S$ is a nonempty admissible peak set and that $m=\max(S)$.
Let $\cQ_S(n):=\{\pi \in \fS_n : S \subseteq P(\pi)\}$ be
the set of permutations $\pi \in \fS_n$ whose peak set contains
$S=\{i_1<i_2<\dots<i_s\}$, and let us partition $S$ into runs of
alternating substrings.  An \textit{alternating substring} is a maximal size subset
$A_r$ such that $A_r = \{i_r, i_r + 2, \dots, i_r + 2(k-1)\} \subseteq
S$, where $i_r - i_{r-1} \ge 3$ if $i_{r-1} \in S$, and we call $A_r$
an alternating substring because
\[
	\pi_{i_r-1} < \pi_{i_r} > \pi_{i_r+1} < \pi_{i_r+2} >
	  \cdots < \pi_{i_r+2(k-1)} > \pi_{i_r+2(k-1)+1}
\]
is an alternating permutation in $\fS_{2k+1}$ under an order-preserving map.
Alternating permutations have peaks at every even index,
and there are $E_{2k+1}$ of them in $\fS_{2k+1}$.
The numbers $E_{2k+1}$ are the tangent numbers
given by the generating function
\begin{align*}
	\tan x &= \sum_{k=0}^\infty \frac{E_{2k+1}}{(2k+1)!} x^{2k+1}\\
				 &= x + \frac{1}{3}x^3 + \frac{2}{15}x^5 + \frac{17}{315}x^7 + \dots
\end{align*}
In 1879, Andr\'e proved this result in \cite{andre} using a generating
function that satisfies a differential equation.  See
\cite{stanley.alt} for more background on alternating permutations.

Now let $\cA(S)$ be the partition of an admissible set $S$
into maximal alternating substrings.
For example, if $S = \{2,5,9,11,19,21,23,26\}$,
then
\[\cA(S) = \{A_1, A_2, A_3, A_5, A_8\} = 
                 \{\{2\}, \{5\}, \{9, 11\}, \{19,21,23\}, \{26\} \}.
\]
The following results demonstrate how we can use $\cQ_S(n)$ to enumerate
permutations with a given peak set.

\begin{lemma}\label{ExpectedValue}
For $n \ge m+1$, we have
\[
	\left|\cQ_S(n)\right| = 
		n!\prod_{A_r \in \cA(S)}\frac{E_{2\left|A_r\right|+1}}{(2\left|A_r\right|+1)!}.
\]
\end{lemma}

\begin{proof} The formula is easily checked in the case $S=\emptyset$,
so assume $S \neq \emptyset$.  Assume the theorem is true by induction
for all sets $S'$ such that $|\cA(S')|<|\cA(S)|$.  Say $A_1 = \{i_1,
i_1 + 2, \dots, i_1 + 2(k-1)\} \in \cA(S)$.  We count the number of
permutations $\pi \in \fS_n$ such that $A_1 \subseteq P(\pi)$ by
choosing $2k+1$ of the $n$ elements, arranging them such that their
peak set is $A_1$ in $E_{2k+1}$ ways, then appending any permutation
of the remaining $n-(2k+1)$ elements arranged to have peak set
contained in $S'=S\setminus A_{1}$.  The result now follows by
induction.  

\end{proof}

\begin{lemma}\label{InclusionExclusion}
For $n \ge m + 1$, we have
\[
	|\cP_S(n)| = \sum_{T \supseteq S}(-1)^{|T-S|}|\cQ_T(n)|.
\]
\end{lemma}

\begin{proof}
The proof follows the inclusion-exclusion principle.
\end{proof}

Call an index $i$ a \textit{free index} of peak set $S$ if $i \in [m + 2]$ and
$i$ is neither a peak nor adjacent to a peak in  $S$.
The following theorem gives us a closed-form expression of tangent numbers
for $|\cP(m+1)|$ and $|\cP(m+2)|$ when $S$ has no free indices.
Note that if $S$ has no free indices, then it can be thought of as
separate independent alternating permutations that are concatenated to each other,
similar to the independence in Theorem~\ref{Gap3Split}.

\begin{corollary}\label{NoFreeIndices}
If $S$ has no free indices and $k\in[2]$, then
\[
	|\cP_S(m+k)| = (m+k)!\prod_{A_r\in\cA(S)}\frac{E_{2|A_r|+1}}{(2|A_r|+1)!}.
\]
\end{corollary}

\begin{proof}
We observe that $S$ is the only admissible superset of $S$ and use
Lemma~\ref{ExpectedValue} and Lemma~\ref{InclusionExclusion}.
\end{proof}

\section{Related work and conjectures}
\label{ConjecturesSection}
In this final section, we relate our work to other recent results about permutations
with a given peak set,
and we also restate some conjectures that stemmed from our work.
Kasraoui characterized in \cite{kasraoui} which peak sets $S$ maximize $|\cP_S(n)|$
for $n \ge 6$ and explicitly computed $|\cP_S(n)|$ for such sets $S$.
We compute the maximum $|\cP_S(n)|$ in a different way using alternating permutations.

\begin{theorem}[{\cite[Theorem 1.1, 1.2]{kasraoui}}]\label{most_freq_S}
For $n \ge 6$, the sets $S$ that maximize $|\cP_S(n)|$ are
\[
	S = 
	\begin{cases}
		\{3,6,9,\dots\} \cap [n-1] \text{ and } \{4,7,10,\dots\} \cap [n-1] &
			\text{if } n \equiv 0\pmod 3,\\
		\{3,6,9,\dots, 3s, 3s+2, 3s+5,\dots\} \cap [n-1]
			\text{ for }1\le s \le \lfloor \frac{n}{3} \rfloor &
			\text{if } n \equiv 1\pmod 3,\\
		\{3,6,9,\dots\} \cap [n-1] & \text{if } n \equiv 2\pmod 3.
	\end{cases}
\]
\end{theorem}

\begin{theorem}[{\cite[Theorem 1.2]{kasraoui}}]
Suppose $n \ge 6$ and $S$ maximizes $|\cP_S(n)|$.
Set $\ell = \lfloor \frac{n}{3} \rfloor$.
Then we have
\[
	|\cP_S(n)| = 
	\begin{cases}
		\frac{1}{5}3^{2-\ell}n! & \text{if } n \equiv 0\pmod 3,\\
		\frac{2}{5}3^{1-\ell}n! & \text{if } n \equiv 1\pmod 3,\\
		3^{-\ell}n! & \text{if } n \equiv 2\pmod 3.
	\end{cases}
\]
\end{theorem}

\begin{proof}[Alternative proof]
We work by cases using Theorem~\ref{most_freq_S}.
When $n \equiv 0 \pmod 3$, there is only one admissible superset of $S$, which
we call $T$.
Using Theorem~\ref{ExpectedValue} and Lemma~\ref{InclusionExclusion}, 
\begin{align*}
	|\cP_S(n)| &= |\cQ_S(n)| - |\cQ_{T}(n)|\\
						 &= n! \left(\frac{1}{3}\right)^{\ell - 1} 
						 		- n!\left(\frac{1}{3}\right)^{\ell - 2}\left(\frac{2}{15}\right)\\
						 &= \frac{1}{5} 3^{2 - \ell} n!,
\end{align*}
as desired.
We use Corollary~\ref{NoFreeIndices} to prove the
cases $n \equiv 1,2 \pmod 3$, which
are simpler because there are no admissible supersets of $S$.
\end{proof}

Another new result in \cite{signed_peaks}
shows that the number of permutations with the same peak
set for signed permutations can be enumerated using the peak polynomial
$p_S(x)$ for unsigned permutations.
We present an alternate proof that can be used to reduce
many signed permutation statistic problems to unsigned permutation statistic problems.
We denote the group of signed permutations as $B_n$.

\begin{theorem}[{\cite[Theorem 2.7]{signed_peaks}}]\label{signed_poly_same}
Let $|\cP_S^*(n)|$ be the number of signed permutations $\pi\in B_n$
with peak set $S$.
We have $|\cP_S^*(n)| = p_S(n)2^{2n-|S|-1}$, where $p_S(x)$ is the same
peak polynomial used to count unsigned permutations $\pi\in \fS_n$ with
peak set $S$.
\end{theorem}

\begin{proof}[Alternative proof]
We naturally partition $B_n$ by the signage of the permutations,
which gives $2^n$ copies of $\fS_n$ under an order-preserving map,
and then we work in each copy of $\fS_n$ separately.
For example, 
$
	B_3 = \fS_{+++} \cup \fS_{++-}\cup \fS_{+-+}\cup \fS_{+--}\cup 
					\fS_{-++}\cup  \fS_{-+-}\cup  \fS_{--+}\cup  \fS_{---}
$, where $\fS_{++-}$ is the set of permutations of $\{1,2,-3\}$.
It follows that $|\cP_S^*(n)| = 2^n |\cP_S(n)|$,
so $|\cP_S^*(n)| = p_S(n)2^{2n - |S| - 1}$ by Theorem~\ref{MainEnumerationTheorem}.
\end{proof}

Now we restate some conjectures.
In the data set \cite{fahrbach}, we experimentally 
checked Conjecture~\ref{BoundedRootsConjecture_conjectures_section}
for all admissible peak sets $S$ where $\max(S) \le 15$.
This conjecture
implies the truth of Conjecture~\ref{PositivityConjecture}, which we explained in
Section~\ref{PositivityConjectureSection}.
We have also shown in Subsection~\ref{sub3-2} that the peak sets listed in
Conjecture~\ref{all-integral-roots-conjectures-section} have only integral roots,
but we have not proven the other direction.
Conjecture~\ref{final_real_roots} is an observation that is related to
Conjecture~\ref{BoundedRootsConjecture_conjectures_section}, and 
we have proved it for all integers $x_0$ using Lemma~\ref{p_in_final_gap}
and Lemma~\ref{strictly_increasing_after_m}, but not all real $x_0$.

\begin{conjecture}\label{BoundedRootsConjecture_conjectures_section}
The complex roots of $p_S(z)$ lie in $\{z\in\mathbb{C} : |z|\le m \text{ and Re}(z)\ge -3 \}$ if $S$ is admissible.
\end{conjecture}

\begin{conjecture}\label{all-integral-roots-conjectures-section}
If $S=\{i_1<i_2<\cdots<i_s\}$ is admissible and all of the roots of $p_S(x)$ are real, then all of the roots of $p_S(x)$ are integral.
Furthermore, $p_S(x)$ has all real roots if and only if $S=\{2\}$, $S=\{2,4\}$, $S=\{3\}$, $S=\{3, 5\}$, $S=\{i_1<i_2<\cdots<i_s<i_s+3\}$, or $S=\{i_1<i_2<\cdots<i_s<i_s+3<i_s+5\}$.
\end{conjecture}

\begin{conjecture}\label{final_real_roots}
Let $S$ be admissible and $|S| \ge 2$.
If $p_S(x_0)=0$ for $x_0\in\mathbb{R}$, then ${x_0>\max(S_1)}$
if and only if $x_0 = \max(S)$.
\end{conjecture}

\begin{question}
What does $p_{S}(n)$ count for $n > \max(S)$?
\end{question}

\section{Acknowledgments}
We would like to thank Jim Morrow first and foremost for organizing the University
of Washington Mathematics REU for over 25 years.
We also would like to thank Ben Braun, Tom Edwards, Richard Ehrenborg, Noam Elkies,
Daniel Hirsbrunner, Jerzy Jaromczyk, Beth Kelly, William Stein,
and Austin Tran for their discussions with us about various results in this paper.
We credit SageMath \cite{sage} and the Online Encyclopedia of Integer Sequences \cite{oeis} for assisting us
with our research.
Finally, we would like to thank the anonymous reviewer for their helpful suggestions.


\end{document}